\documentclass[a4wide,12pt]{article}

\usepackage{amsmath}
\usepackage{amssymb}
\usepackage{amsfonts}
\usepackage{amsthm}
\usepackage{mathtools}
\usepackage{mathrsfs}
\usepackage{bm}
\usepackage{dsfont}
\usepackage{latexsym}

\usepackage{microtype}
\usepackage{setspace}
\usepackage{indentfirst}
\usepackage{textcomp}
\usepackage{geometry}
\usepackage{lineno}

\usepackage{graphicx}
\usepackage{epstopdf}
\usepackage{float}
\usepackage{subcaption}
\usepackage{multirow}

\usepackage{xcolor}

\usepackage{enumerate}
\usepackage[shortlabels]{enumitem}

\usepackage{lscape}

\usepackage{calc}
\usepackage{comment}
\usepackage{listings}

\usepackage{tikz}
\usetikzlibrary{plotmarks}
\usepackage{pgfplots}

\usepackage{algorithmic}
\usepackage[algoruled,boxed,lined]{algorithm2e}

\usepackage{cite}

\usepackage[hidelinks]{hyperref}
\hypersetup{
  colorlinks = true,
  urlcolor   = blue,
  linkcolor  = red,
  citecolor  = red
}

\usepackage{cleveref}

\usepackage{fancyhdr}
\fancypagestyle{plain}{%
  \fancyhf{}
  \fancyfoot[C]{\thepage}
  
}

\ifpdf
  \DeclareGraphicsExtensions{.eps,.pdf,.png,.jpg}
\else
  \DeclareGraphicsExtensions{.eps}
\fi

\title{Chaos propagation in genetic algorithms:\\ An optimal transport approach}

\author{Giacomo Borghi\thanks{Maxwell Institute for Mathematical Sciences and Department of Mathematics, School of Mathematical and Computer Sciences (MACS), Heriot-Watt University, Edinburgh, UK
  (\texttt{g.borghi@hw.ac.uk}).}
}

\usepackage{amsopn}

\usepackage{mathrsfs}
\usepackage{mathtools}
\usepackage{amssymb}
\usepackage[shortlabels]{enumitem}
\usepackage{bm}
\usepackage{yfonts}
\usepackage{eufrak}

\newcommand{\rev}[1]{#1}

\newcommand{\Rd}{\mathbb{R}^d}
\newcommand{\RR}{\mathbb{R}}
\newcommand{\x}{\mathbf{x}}
\newcommand{\w}{\mathbf{w}}
\newcommand{\unif}{\mathrm{Unif}}
\newcommand{\law}{\mathrm{Law}}
\newcommand{\lip}{\mathrm{Lip}}
\newcommand{\ve}{\varepsilon}
\newcommand{\W}{\mathbb{W}}
\newcommand{\WW}{\mathbb{W}_{1}^{\wedge}}
\newcommand{\F}{\mathbf{F}}
\renewcommand{\d}{\mathrm{d}}

\newcommand{\BL}{\mathtt{BL}}

\newcommand{\OX}{\overline{X}}
\newcommand{\X}{\mathbf{X}}
\renewcommand{\j}{\mathbf{j}}
\newcommand{\D}{{D^\wedge}}

\newtheorem{theorem}{Theorem}[section]
\newtheorem{lemma}[theorem]{Lemma}

\theoremstyle{definition}
\newtheorem{definition}[theorem]{Definition}

\theoremstyle{remark}

\newtheorem{assumption}[theorem]{Assumption}

\numberwithin{equation}{section}

\ifpdf
\hypersetup{
  pdftitle={},
  pdfauthor={G. Borghi}
}
\fi

\begin{document}

\maketitle

\begin{abstract}
Genetic algorithms are high-level heuristic optimization methods which enjoy great popularity thanks to their intuitive description, flexibility, and, of course, effectiveness. The optimization procedure is based on the evolution of possible solutions following three mechanisms: selection, mutation, and crossover. In this paper, we look at the algorithm as an interacting particle system and show that it is described by a Boltzmann-type equation in the many particles limit. Specifically, we prove a propagation of chaos result with a novel technique that leverages the optimal transport formulation of the \rev{bounded Lipschitz} norm and naturally incorporates the crossover mechanism into the analysis. \rev{The convergence admits a rate with respect to the number of particles, corresponding to the optimal rate in the Wasserstein-1 distance.} 
\end{abstract}

\bigskip

\textbf{keywords:}
Genetic algorithms, interactive particle systems, Boltzmann equation, propagation of chaos, Wasserstein distance, Nanbu method

\bigskip
\textbf{MSCcodes:}
82C22,
35Q20,
65C05,
90C59,
49Q22


\section{Introduction and main result}

Evolutionary algorithms \cite{EibenSmith2015} are a class of  solvers for black-box optimization problems where the lack of sufficient structure makes it difficult to employ classical mathematical programming techniques like gradient-based optimizers. 
For instance, the objective function may be non-differentiable, or extremely noisy, or the search space may be non-Euclidean or not even a manifold. In these scenarios, where searching for solutions at random seems the only viable option, evolutionary computation offers a more effective strategy inspired by nature.

Current prominent applications of evolutionary algorithms are in  training and inference of artificial intelligence models. In Population-Based Training \cite{jaderberg2017Population} a class of neural networks collective train by copying and mutating the best performing ones.  In Evolutionary Model Merge \cite{akiba2025evolutionary} models are selected and merged to transfer knowledge across domains without re-training. Model Soups \cite{wortsman2022model}  also merges the weights of different models to improve hyperparameter search. In Mind Evolution \cite{lee2025Evolving} candidate responses generated by large language models are iteratively mutated and refined to improve accuracy and reduce inference time.

The Genetic Algorithm (GA) \cite{goldberg1989,holland1992adaptation,katoch2021} is a paradigmatic example of evolutionary computation including the three main strategies that are used to refine (or \textit{evolve}) a set of solutions (or \textit{particles}):
\begin{itemize}
\item \textbf{selection:} best solutions are selected for reproduction;
\item \textbf{crossover:} new offspring are created by mixing two parent solutions;
\item \textbf{mutation:} newly created solutions are stochastically perturbed.
\end{itemize}

We will consider in this paper a GA for continuous optimization \cite{rawlins1991real} over the search space $\Rd$ where the aim is to maximize a given fitness function $\F:\Rd \to \RR$. Given a set of $N$ possible solutions $\{x^i\}_{i=1}^N$,  $x^i \in \Rd$ is selected for reproduction proportionally to $\F(x^i)$. For two parent solutions $(x,x_*)$, the offspring is created as 
\begin{equation}\label{eq:coll}
x' = \underbrace{(1 - \gamma)\odot x + \gamma \odot x_*}_{\textup{crossover}}\; + \!\underbrace{\sigma\xi}_{\textup{mutation}}
\end{equation}
where $\gamma\sim \mu^\gamma$ is a random crossover vector taking values in $[0,1]^d$, $\odot$ is the component-wise product, $\xi\sim \mu^{\xi}$ is a random mutation vector in $\Rd$, typically with zero-mean and identity covariance, and $\sigma>0$ determines the mutation strength. An additional parameter $\tau \in (0,1]$ determines the probability of a solution to be substituted by a new offspring at each iteration. $\tau = 1$ corresponds to a fully synchronous update, where each solution is updated, while with $\tau = 1/N$ only one solution is expected to update per iteration on average. The precise algorithmic strategy is described in Algorithm \ref{alg:ga}.

\subsection{Main result} 

The GA algorithm can be understood as a particle system undergoing binary interactions of jump type, similarly to physical particles in rarefied gases undergoing binary collisions. Such systems can be statistically described by a probability distribution $f = f(t,x) \in \mathcal{P}(\Rd), t\geq0,$ evolving according to an integro-differential equation of Boltzmann type \cite{graham1997stochastic}
\begin{equation}\label{eq:boltzmann}
\frac{\partial f}{\partial t} = Q^+(f,f) - f
\end{equation}
where $Q^+$ is the so-called gain operator defined as 
\begin{equation}\label{eq:gain}
\langle \phi, Q^+(f,f)\rangle  = \iiiint \frac{\F(x)\F(x_*)}{\langle \F,  f\rangle^2} \phi(x') f(\d x) f(\d x_*) \mu^\gamma(\d \gamma) \mu^{\xi}(\d \xi)
\end{equation}
for any bounded test function $\phi$. 

Our objective is to quantify this connection by showing that as $N\to \infty$ and $\tau \to 0$, the random empirical measure constructed by Algorithm \ref{alg:ga} 
\[
f_n^N = \frac1N\sum_{i=1}^N \delta_{X^i_n}
\]
converges to measure solutions of \eqref{eq:boltzmann} with a rate. To quantify convergence, we consider the \rev{Bounded-Lipschitz (BL) norm \cite{dudley2018real}} for signed Radon measures
\rev{\begin{equation} \label{eq:BL}
\| \mu \|_\BL:= \sup \left \{\int \phi(x)\mu(\d x)\, \middle |\, \textup{Lip}(\phi)\leq 1, \, \|\phi\|_{\infty}\leq \frac12 \right\}\,\qquad \mu\in\mathcal{M}(\Rd)\,,
\end{equation}
also known as flat metric between measures.}

\begin{assumption} \label{asm:f0}
The initial distribution $f_0$ has $q$ finite moments, $f_0 \in \mathcal{P}_q(\Rd)$ with $q > 2$
for $d \leq 2$, and $q \geq d/(d-1)$ for $d > 2$.
\end{assumption}

At time 0, we know particles are independent and $f_0$-distributed. Under Assumption \ref{asm:f0} we can quantify the expected error as 
\begin{equation}
\mathbb{E}\| f^N_0 - f_0 \|_\rev{\BL} \lesssim M_q^{1/q}(f_0)\ve_1(N)
\end{equation}
where $M_q(f_0) := \int |x|^q f_0(\d x)$, and $\ve_1(N)\to 0$ as $N\to \infty$, is a rate which only depends on $N$. We refer to  \cite{fournier2015rate}, or \rev{Lemma \ref{l:rate}}, for more details on this concentration inequality. 

The main result of the paper is showing that the convergence rate $\ve_1(N)$ is maintained  also at subsequent times, with an additional error of order $\mathcal{O}(\tau)$ given by the time discretization of the PDE. This means that particles generated by Algorithm \ref{alg:ga} 
become independent on one another as $N\to \infty$, showing chaotic behavior.

\rev{\begin{assumption}
\label{asm:F} The fitness function is globally Lipschitz continuous, bounded, and strictly positive. Specifically, it satisfies $\lip(\F) := \sup_{x\neq y}|\F(x) - \F(y)|/|x - y|< +\infty$, and 
$\F : \Rd \to [\underline{\F}, \overline{\F}]$ with $\overline{\F}>\underline{\F}>0$.
\end{assumption}}

\begin{theorem}
\label{t:main}
Let $f_0$ and $\F$ satisfy Assumptions  \ref{asm:f0} and \ref{asm:F}, $\mu^\gamma = \textup{Unif}[0,1]^d$, and $\mu^\xi = \mathcal{N}(0,I_d)$. Let $f\in C([0,T],\mathcal{P}_q(\Rd))$ be a measure solution to \eqref{eq:boltzmann} in duality with $C_b(\Rd)$ (see Definition \ref{def:solution}) with initial data $f_0$, and time horizon $T>0$.

\rev{The empirical measure $f^N_n$ constructed by Algorithm \ref{alg:ga}
satisfies
\[
\sup_{t_n = n\tau\leq T} \mathbb{E} \| f^N_n - f(t_n)\|_{\BL} \leq C  T e^{C'  T}  \left( \,  \ve_1(N) \,+ \,\tau\, \right)
\]
with $C = C(\lip(\F), \underline{\F},\overline{\F}, d,q, M_q(f_0)^{1/q}, \sigma), C' = C'(\lip(\F), \underline{\F},\overline{\F}, q)$ positive constants. }
\end{theorem}

Notation and preliminary lemmas are presented in Section \ref{sec:pre}, while the proof is given in Section \ref{sec:main}. First, though, we provide more context on the analysis of GA algorithms and the novelty of our approach.

\begin{algorithm}[t]
\caption{Genetic Algorithm}
\label{alg:ga}
\begin{algorithmic}
\STATE{\textbf{Input:} Fitness function $\mathbf{F}$, population size $N$, initial distribution $f_0\in \mathcal{P}(\Rd)$, $\tau \in (0,1]$, maximum iterations $n_{\max}$.}
\STATE{Initialize $N$ solutions $X^i_0\sim f_0$, $i = 1, \dots,N$}
\STATE{$n = 0$}
\WHILE{$n \leq n_{\max}$}
\FORALL{$i = 1, \dots, N$}
\STATE{with probability $\tau:$}
\STATE{\hspace{0.8cm}  sample first parent solution $X$ according to
$\mathbb{P}(X = X^j_n) = \frac{\F(X^j_n)}{\sum_{\ell=1}^N \F(X^\ell_n)}$}  
\STATE{\hspace{0.8cm}  sample second parent $X_*$ in the same way}
\STATE{\hspace{0.8cm}  sample $\gamma \sim \textup{Unif}[0,1]^d$, $\xi \sim \mathcal{N}(0, I_d)$}
\STATE{\hspace{0.8cm}  $\tilde{X}^i_{n} = (1 - \gamma)\odot X + \gamma \odot X_* + \sigma \xi$}
\STATE{else:}
\STATE{\hspace{0.8cm}  $\tilde{X}_{n}^i = X_{n}^i$ }
\ENDFOR
\STATE{$X^i_{n+1} = \tilde{X}^i_n$ for all $i = 1,\dots,N$}
\STATE{$n = n+1$}
\ENDWHILE
\end{algorithmic}
\end{algorithm}

\subsection{Literature review and novelty}
The connection between interacting particle systems and genetic algorithms   has been extensively studied in the series of works \cite{delmoral1999stability, delmoral2000branching,delmoral2000moran,delmoral2001modeling}. GA is viewed as a numerical Monte Carlo approach to solve Feynman--Kac formulae \cite{delmoral2004book} which is a class of more general dynamics also arising in nonlinear filtering methods and Sequential Monte Carlo samplers \cite{delmoral2006sequential}. To obtain this interpretation, only the selection-mutation mechanism (without crossover) is considered.  The dynamics is studied both in discrete and continuous time settings, see for example \cite{delmoral2000branching}. The authors address both stability with respect to the many particle limit $N\to\infty$, and long-time stability of the resulting measure-valued iteration. Such convergence analysis is typically based on martingales and semi-group techniques, the use of \rev{Total Variation norm}, and Zolotarev semi-norms \cite{rachev1991probability, delmoral2025kantorovich}. 
Oftentimes the fitness is assumed to be bounded as in Assumption \eqref{asm:F}, see for instance \cite[Condition ($\mathcal{G}$), Eq. (22)]{delmoral2000branching}.
The mechanism of crossover, where two selected particles are mixed to create a new one, is to our knowledge only considered in \cite{delmoral2001modeling} by doubling the state space (in our setting, $\Rd \times \Rd$) in order not to break the Feynman--Kac-type description. 

The analysis we present here follows a complementary approach based on optimal transport theory which naturally includes the crossover mechanism as a binary interaction. It is related to the propagation of chaos \cite{sznitman1991topics,diez2022review1,diez2022review2} properties of Nanbu-type Monte Carlo schemes   for the homogeneous Boltzmann equation \cite{nanbu1980}. As in  \cite{borghi2025kinetic}, we interpret the crossover-mutation mechanism as a collision between particles. This connection was already noted in \cite{delmoral2000moran}, but not used in the convergence analysis.  Chaos propagation for generalized Boltzmann models have been studied, for instance, in  \cite{graham1997stochastic,cortez2016}, but with a linear selection kernel, and not a nonlinear one as it appears in \eqref{eq:boltzmann}. 

We provide a concise proof of Theorem \ref{t:main} by exploiting both the primal and dual formulation of \rev{a $L^1$-Wasserstein distance  $\WW$ with cost given by 
\begin{equation}\label{eq:dist}
\D(x,y) := |x- y| \wedge 1 = \min\{|x - y|, 1\}\,,
\end{equation}
which coincides with the BL metric, $\WW(f,g): = \|f - g\|_{\BL} $, between two probability measures $f,g$, see Section \ref{sec:pre} for more details. The distance $\D$ sits in between $|x - y|$ and $\bm{1}_{x\neq y} = \delta(x - y)$, leading to a metric which inherits benefits from both the usual Wasserstein-1 metric, and the Total Variation norm.
With the optimal transport formulation, we treat the crossover-mutation mechanism, thanks to the Lipschitzianity of the collision \eqref{eq:coll}, while with the dual one we show stability of the selection procedure.} The technique allows to optimally couple the parent particles at each step, similar to the technique employed in \cite{fournier2016nanbu,cortez2016}, so that the  convergence rate of i.i.d. particles is not lost. The key argument is Lemma \ref{l:coupling} which is a modified version of \cite[Lemma 4.1]{borghi2025nanbu} for weighted particles. \rev{The distance \eqref{eq:dist} has been used in \cite{hairer2008spectral,fontbona2022quantitative,cao2026quantitative}, while generalizations beyond probability measures have been considered in \cite{piccoli2014generalized,piccoli2019signed}.} 

As the Nanbu Monte Carlo method is based on the explicit Euler discretization of the Boltzmann equation, this is what Algorithm \ref{alg:ga} converges to as $N\to \infty$. Therefore, our propagation of chaos result also holds in the time-discrete settings considered in \cite{delmoral1999stability, delmoral2000branching,delmoral2001modeling} by considering $\tau = 1$ ($\tau$ being the time step of the Euler discretization). 
In Theorem \ref{t:main}, we also included convergence  to the time continuous PDE as $\tau \to 0$  for completeness. The latter result though is more standard, as well as the existence of solutions to Boltzmann-type equations \cite{mischler1999euler,horsin2003convergence,borghi2025nanbu}, which instead we omitted for brevity. Finally, we mention that convergence of \eqref{eq:boltzmann} towards global solutions of the optimization problem $\max_x \F(x)$  has been studied in \cite{borghi2025kinetic}. 

\section{Notation and preliminaries}
\label{sec:pre}

Random variables are assumed to be defined with respect to an underlying probability space $(\Omega, \mathcal{F}, \mathbb{P})$. $X\sim f$ means that $X$ is a random variable with law $f$. 
For a given Borel set $A\subseteq\Rd$, $\textup{Unif}(A)$ is the uniform probability measure over $A$. For $\tau \in (0,1)$, we have $\textup{Bern}(\tau) = (1 - \tau) \delta_0 + \tau \delta_1$ with $\delta_x$ the Dirac delta measure centred at $x$.  The indicator function for a given set $A$ is $\bm{1}_A$. We will sometimes denote $\{1,\dots, N\}$ with $[N]$, and $\Delta_N$ is the unit simplex  $\Delta_N = \{\w\in \RR^N\,|\,\sum_i w^i = 1, w\geq0 \}$, with $\textup{int}(\Delta_N) = \Delta_N \cap \RR_{>0}^N$ the interior points where weights are strictly positive. \rev{We use $a \wedge b : = \min(a,b)$, and $a\vee b = \max\{a,b\}$ for the minimum and maximum operations}.

\subsection{Wasserstein distances}

With $\mathcal{P}(\Rd)$ we denote the set of Borel probability measures over $\Rd$, and for $\phi: \Rd \to \RR$ we write $\langle \phi, f\rangle = \int_{\Rd} \phi(x) f(\d x)$.
Let $M_p(f):=\int |x|^p f(\d x), p \geq 1$ be the $p$-th moment of $f$, and  $\mathcal{P}_p(\Rd)$ be the subset of probabilities with finite $p$-th moments.
\rev{For a given cost function $c$, and $p \geq 1$, we consider the Wasserstein distances \cite{villani2009}
\[
\W_{p}^c(f,g) =  \left ( \inf_{X\sim f, Y\sim g}\mathbb{E} \left[c(X,Y)^p\right] \right)^{1/p}.
\]%
We write $f\in C([0,T],\mathcal{P}_p(\Rd))$ if the curve $f(t), t\in[0,T],$ is continuous with respect to $\W_p^{|\cdot|}$ (the usual $L^p$-Wasserstein distance). 

Let $\textup{Lip}^{d}(\phi) := \sup_{x\neq y} |\phi(x) - \phi(y)|/d(x,y)$ denote the Lipschitz constant of a function $\phi$ with respect to a distance $d$. By Kantorovich--Rubinstein duality \cite[Theorem 1.14]{villani2009}, when the cost  $c$ is a distance, it holds
\[
\W_{1}^c(f,g) = \sup\left \{  \langle \phi, f - g\rangle\, \middle |\, \textup{Lip}^c(\phi) \leq 1\right \}\,.
\]
If $c(x,y) = |x - y|$ one gets that $\W_1^{|\cdot|}$ corresponds to the Kantorovich--Rubinstein (KR) norm (duality with $\lip(\phi) := \lip^{|\cdot|}(\phi)\leq 1$), while if $c(x,y) = \delta(x - y) = \bm{1}_{x \neq y}$,  $\W^\delta_1$ corresponds the Total-Variation (TV) norm (duality with $\|\phi\|_\infty := \sup_x\phi(x) \leq 1$), see, for instance, \cite[Section 3]{piccoli2016properties}.

For the distance function \eqref{eq:dist} of our interest, $\D(x,y) = |x - y| \wedge 1$, condition $\lip^\D(\phi)\leq 1$ leads to 
\[
|\phi(x) - \phi(y)| \leq \D(x,y) = |x-y|\wedge 1\,\qquad \textup{for all}\quad x, y \in \Rd\,.
\]
Therefore, a test function $\phi$ needs to satisfy $\lip(\phi) = \lip^{|\cdot|}(\phi) \leq 1$, and  $\textup{osc}(\phi) := \sup\phi - \inf\phi \leq 1$. Since $\langle \phi, f - g\rangle$ is invariant under translations, condition $\textup{osc}(\phi)\leq 1$ is equivalent to  $\|\phi \|_{\infty}\leq 1/2$, as in the definition of the BL norm \eqref{eq:BL}.
Therefore, we obtain that for 
\[\WW(f,g)  := \W^\D_1(f,g) = \inf_{X\sim f, Y\sim g}\mathbb{E}\left[|X - Y|\wedge 1 \right]\]
it holds
\[\WW(f,g) = \|f - g\|_\BL\,.\]
Moreover, since $\D(x,y) \leq |x - y|$, we also have $\WW \leq \W_1^{|\cdot|}$. Generalisations to this optimal transport formulation of the BL norm have been studied in \cite{piccoli2014generalized,piccoli2019signed}.}


\subsection{Auxiliary lemmas}

\rev{We collect some definitions and auxiliary results that will be used throughout the paper.

\begin{definition}[Measure solution to \eqref{eq:boltzmann}] \label{def:solution}
Given $ T\in (0, \infty)$, we say $f\in C([0,T], \mathcal{P}_q(\Rd))$ is a measure solution to \eqref{eq:boltzmann} with initial data $f_0\in \mathcal{P}_q(\Rd)$ if for any $\phi \in C_b(\Rd)$ we have $\lim_{t\to 0^+} \langle \phi, f(t)\rangle = \langle \phi, f_0\rangle$ and
\[
\frac{\d}{\d t}\langle \phi, f(t) \rangle = \langle \phi, Q^+(f(t), f(t))\rangle - \langle \phi, f(t)\rangle \qquad \textup{for a.e.}\quad t \in (0,T)\,.
\]
\end{definition}}

We recall the empirical concentration inequality from \cite[Theorem 1]{fournier2015rate}, in the specific case of \rev{Wasserstein distance $\W^{|\cdot|}_1$}, and under the assumption of having sufficient finite moments. \rev{Note that since $\WW\leq \W_1^{|\cdot|}$, the same inequalities holds of the BL norm as well.}

\begin{lemma}[Concentration rate] \label{l:rate}
Let $f\in \mathcal{P}_q(\Rd)$ with $q>0$ as in Assumption \ref{asm:f0},
and $f^N = (1/N) \sum_{i=1}^N \delta_{X^i}$ with $X^i\sim f$ i.i.d.. There exists a constant $C = C(d,q)>0$ such that for all $N\geq1$:
\[
\mathbb{E}\rev{\W^{|\cdot|}_1}(f,f^N) \leq C M_q^{1/q}(f)\ve_1(N)
\]
with 
\[
\ve_1(N) := 
\begin{cases}
N^{-1/2} & \textup{if}\;\; d = 1\,, \\
N^{-1/2}\log(1 + N) & \textup{if}\;\; d = 2\,, \\
N^{-1/d} & \textup{if}\;\; d > 2 \,.\\
\end{cases}
\]
\end{lemma}
\rev{
\begin{lemma}[Crossover-mutation stability]\label{l:lipschitz}
For any fixed $\gamma\in [0,1]^d$ and $\sigma>0, \xi\in \Rd$, the crossover-mutation update \eqref{eq:coll} is Lipschitz in  $x,x_*$ with respect to the metric $\D$.
\end{lemma}
\begin{proof}
Since $\gamma\in [0,1]^d$, we have $|\gamma \odot z| \leq |z|$ for any $z\in \Rd$. Then, consider $x'$ and $y'$ defined as in \eqref{eq:coll}, we have
\begin{align*}
| x' - y'| &  = \left |(1 - \gamma)\odot x + \gamma \odot x_* + \xi \sigma - \Big((1 - \gamma)\odot y + \gamma \odot y_* + \xi \sigma\Big) \right| \\
& = |(1- \gamma) \odot (x - y) + \gamma \odot(x_* - y_*)| \leq |x - y| + |x_* - y_*|\,.
\end{align*} 
Then, Lipschitzianity with respect to $\D(x,y) = |x - y|\wedge 1$ follows by
\[
 |x' - y'|\wedge 1 \leq (|x - y| + |x_* - y_*|)\wedge 1 \leq |x - y|\wedge 1 + |x_* - y_*|\wedge 1 \,.
\]
\end{proof}}%
\rev{
\begin{lemma}[Selection stability]
\label{l:stability}
Let $\F$ satisfy Assumption \ref{asm:F}. There exists a positive constant $C_\F = C_\F(\underline{\F}, \overline{\F}, \lip(\F))$ such that  
\[
\WW \left(\frac{\F f}{\langle \F,f\rangle},\frac{\F g}{\langle \F,g\rangle}\right) \leq C_\F \WW(f,g)
\quad \textup{for any}\quad f,g\in \mathcal{P}(\Rd)\,.
\]
\end{lemma}
\begin{proof}
We use the dual formulation \eqref{eq:BL} of $\WW$, and so consider a test function $\phi$ with $\lip(\phi)\leq 1$ and $\|\phi\|_{\infty}\leq 1/2$. By triangular inequality we have
\begin{align*}
\left\langle \phi, \frac{\F f}{\langle \F,f\rangle} - \frac{\F g}{\langle \F,g\rangle}\right \rangle \leq 
\left | 
\left\langle \phi, \frac{\F f}{\langle \F,f\rangle} - \frac{\F g}{\langle \F,f\rangle}\right \rangle
\right| + 
\left | 
\left\langle \phi, \frac{\F g}{\langle \F,f\rangle} - \frac{\F g}{\langle \F,g\rangle}\right \rangle
\right| =: I_1 + I_2\,.
\end{align*}
Since dividing $\phi \F$ by $\lip(\phi \F) \vee 2\|\phi \F\|_\infty$ turns it into a test function for the BL norm, it holds
\[\langle \phi \F, f - g \rangle \leq \left( \lip(\phi \F) \vee 2\|\phi \F\|_\infty \right)\| f - g\|_{\BL}\,.
\]
Given that $\|\phi \F \|_{\infty} \leq \overline {\F}/2$ and $\lip(\phi \F) \leq \lip(\F)/2 + \overline{\F}$, we get the following bound for $I_1$
\begin{equation*}
I_1 \leq \frac1{\underline{\F}}|\langle \phi, \F f - \F g\rangle |
 \leq \frac{\lip(\phi \F) \vee 2\|\phi \F \|_\infty} {\underline{\F}}\| f -  g\|_{\BL}
 \leq \frac{ (\lip(\F)/2 + \overline{\F}) }{\underline{\F}}\| f -  g\|_{\BL}\,,
\end{equation*}
where we also used $\langle \F, f\rangle \geq \underline{\F}$ in the first inequality.
Using the same type of  estimates, we have for $I_2$
\begin{alignat*}{3}
I_2  &= \left | 
\left\langle \phi, \frac{\F g}{\langle \F,f\rangle} - \frac{\F g}{\langle \F,g\rangle}\right \rangle
\right| 
 & & =  \left( \frac{1}{\langle \F,f\rangle} - \frac{1}{\langle \F,g\rangle}\right ) \left | 
\left\langle \phi,  \F g \right \rangle
\right| \\
& \leq \frac{|\langle \F,g\rangle - \langle \F,f\rangle|}{\langle \F,f\rangle \langle \F,g\rangle} \| \phi \F \|_{\infty} 
&&
\leq \frac{\overline{\F} (\lip(\F) \vee 2\overline{\F}) }{2\underline{\F}^2} \| f - g\|_{\BL}\,.
\end{alignat*}
By collecting the two bounds, and simplifying slightly, one obtains the desired estimate with 
\[C_\F := 
\frac1{\underline{\F}}\left( 1 + \frac{\overline{\F}}{\underline{\F}} \right)(\lip(\F) +  2\overline{\F})\,.
\]
\end{proof}}%

\section{Proof of Theorem \ref{t:main}}
\label{sec:main}

\subsection{Two particle systems} 
The proof is based on a coupling method where we couple the particle system $\X_n = (X^1_n, \dots, X^N_n), n \geq 0,$ generated by Algorithm \ref{alg:ga} with a system of nonlinear i.i.d. particles $\overline{\X} = (\OX_n^1,\dots,\OX^N_n), n \geq 0,$ whose law is given by the explicit time-discretization of the Boltzmann-type equation \eqref{eq:boltzmann}
\begin{equation} \label{eq:euler}
f_{n+1} = (1 - \tau) f_n + \tau Q^+(f_n, f_n)\,.
\end{equation}
We start by re-writing the update of $\X_n$ via auxiliary random variables.

Given the particles' state $\X_n$, to define $\X_{n+1}$ we sample for each particle five independent random variables. The first two are $\gamma^i_n \sim \textup{Unif}[0,1]^d$ and $\xi_n^i\sim \mathcal{N}(0, I_d)$, as in Algorithm \ref{alg:ga}. The third is a Bernoulli random variable $\tau_n^i \sim \textup{Bern}(\tau)$ which determines whether the particle is updated at  step $n$ or not. The last two, $\alpha_n^{i,1}, \alpha_n^{i,2}\sim \textup{Unif}[0,N)$, are used to determine which parent particles generate the new offspring.
Let $\w_n = (w_n^1, \dots, w^N_n)$, $w^i_n = \F(X^i_n)/ \sum_\ell \F(X^\ell_n)$ be the fitness-dependent weights. The parents are determined via the index map
\begin{equation} \label{eq:j}
\begin{split}
\j({\w}, \cdot) &: [0,N) \to \{1,\dots,N\} \\
\j(\w, \alpha)&:=\min\Big\{\,j\in\{1,\dots,N\}: \alpha<N \sum_{\ell=1}^j w^\ell\,\Big\}\,.
\end{split}
\end{equation}
Note that $\mathbb{P}(\j(\w,\alpha) = j) = w^j$ if $\alpha \sim \unif[0,N)$.
The particles evolution described by Algorithm \ref{alg:ga} can then be written compactly as
\begin{equation} \label{eq:nanbu}
X_{n+1}^i = (1 - \tau^i_n) X^i_n + \tau_n^i\left((1-\gamma^i_n) \odot X^{\j(\w_n,\alpha^{i,1}_n)}_n  + \gamma^i_n\odot X^{\j(\w_n,\alpha^{i,2}_n)}_n + \sigma \xi_n^i \right),
\end{equation}
for $i \in [N]$.

Next, we construct the nonlinear auxiliary system $(\overline{\X}_n)_{n\geq0}$ such that all the particles $\OX^1_n, \dots, \OX_n^N$ are i.i.d. and $f_n$ distributed, where $f_n$ is given by \eqref{eq:euler}. To do so, we employ an auxiliary function $X^*_n:[0,N) \to \Rd$ which satisfies the following condition:
\[
X^*_n( \alpha) \sim \frac{\F f_n}{\langle \F, f_n \rangle}\,, \qquad \textup{if}\quad \alpha \sim \textup{Unif}[0,N)\,.
\]
A construction of such a random variable is provided in Section \ref{sec:coupling}. Starting from $\overline{\X}_0= \X_0$, the nonlinear system is then iteratively defined for $i \in [N]$ as 
\begin{equation} \label{eq:iid}
\OX_{n+1}^i = (1 - \tau^i_n) \OX^i_n + \tau_n^i\left((1-\gamma^i_n)\odot X^{*}_n(\alpha_n^{i,1})  + \gamma^i_n\odot X^{*}_n(\alpha_n^{i,2}) + \sigma \xi_n^i \right).
\end{equation}

By computing $\mathbb{E}\phi(\OX^i_{n+1})$ one can directly check that the law of $\OX^i_{n+1}$ satisfies the update rule \eqref{eq:euler}, with $Q^+$ defined as in \eqref{eq:gain}. Independence of the nonlinear particles, instead, follows from the independence of the random variables $\gamma^i_n, \xi_n^i, \tau_n^i, \alpha^{i,1}_n, \alpha^{i,2}_n$ used to define the update.

\subsection{Coupling of parents}
\label{sec:coupling}
To construct $X^{*}_n$, we adapt \cite[Lemma 4.1]{borghi2025nanbu} to our settings, where particles are weighted. 
The fundamental idea is that the two pairs of parents $(X^{\j(\alpha_n^{i,1})}_n, X^{\j(\alpha_n^{i,2})}_n)$ and $(X^*_n(\alpha_n^{i,1}),X^*_n(\alpha_n^{i,2}))$ should be optimally coupled to minimize the Wasserstein distance between them. We will construct $X^*_n(\cdot)$ with this additional aim. The settings of the result are left more general, as the Lemma might be of independent interest.

\begin{lemma}\label{l:coupling}
\rev{For $p\geq 1$, let $\W^c_p$ be the Wasserstein distance with \rev{a lower semi-continuous} cost function
$c(x,y)$, and let $f\in\mathcal{P}(\Rd)$ such that $\W_p^c(f, \delta_0)<+\infty$}. Consider
$\x=(x^1,\dots,x^N)\in(\Rd)^N$ and weights $\w=(w^1,\dots,w^N)\in\textup{int}(\Delta_N)$
such that $w^i>0$ and $\sum_{i=1}^N w^i=1$. Define the weighted empirical measure as
$
\mu_{\x,\w}:=\sum_{i=1}^N w^i\,\delta_{x^i},
$
and recall from \eqref{eq:j} the index-map $\j(\w,\cdot)$.

There exists a measurable mapping
\begin{align*}
X_{f}^{*}:(\RR^d)^N\times \textup{int}(\Delta_N) \times[0,N) &\to \RR^d\\
(\x,\w,\alpha) &\mapsto X_f^{*}(\x,\w,\alpha)
\end{align*}
with the following property: if $\alpha\sim\unif[0,N)$, then the pair
$\big(X_f^{*}(\x,\w,\alpha),\,x^{\mathbf{j}(\w, \alpha)}\big)$ is an optimal coupling
between $f$ and $\mu_{\x,\w}$, namely
\rev{\[
\mathbb{E}\left[c\Big(X_f^{*}(\x,\w,\alpha), x^{\mathbf{j}(\w, \alpha)}\Big)^p\right]
=  \left(\W_p^c\big(f,\mu_{\x,\w}\big)\right)^p.
\]}%
\end{lemma}

\begin{proof} 
Let $\pi_{\x,\w}\in\mathcal{P}(\RR^d\times\RR^d)$ be an optimal transport plan
between $f$ and $\mu_{\x,\w}$ \rev{(assumption $\W_p^c(f, \delta_0)<+\infty$ ensures its existence)}. Thanks to a measurable selection result \rev{for lower semi-continuous costs}, see,
for instance, \cite[Corollary 5.22]{villani2009}, there exists a measurable mapping
\[
(\x,\w)\mapsto \pi_{\x,\w} \quad \textup{s.t.}\quad
\pi_{\x,\w}\in\Gamma_o\big(f,\mu_{\x,\w}\big),
\]
where $\Gamma_o(f,\mu_{\x,\w})$ denotes the set of optimal couplings between
$f$ and $\mu_{\x,\w}$ for the given cost.

Define for any Borel set $A\subseteq\RR^d$
\[
G^i(\x,\w,A):=\frac{\pi_{\x,\w}(A\times\{x^i\})}
{\pi_{\x,\w}(\RR^d\times\{x^i\})}\,.
\]
We note that $G^i$ is a probability kernel from $(\RR^d)^N\times\textup{int}(\Delta_N)$ into $\RR^d$,
thanks to the measurability of $(\x,\w)\mapsto\pi_{\x,\w}$, and so there exists
$g^i=g^i(\x,\w,\beta)$ such that $\law(g^i(\x,\w,\beta))= G^i(\x,\w,\cdot)$
if $\beta\sim\unif[0,1)$, see \cite[Lemma 4.22]{kallenberg2021}. This procedure is
called \textit{randomization} of $G^i(\x,\w,\cdot)$.

Let $S_0:=0$ and $S_i:=\sum_{\ell=1}^i w^\ell$ for $i=1,\dots,N$. \rev{Since weights are strictly positive, define} the rescaled variable
\[
\beta(\w,\alpha):=\frac{\alpha-N S_{\mathbf{j}(\w,\alpha)-1}}{N\,w^{\mathbf{j}(\w,\alpha)}}\in[0,1).
\]
Observe that, if $\alpha\sim\unif[0,N)$, then $\mathbb{P}(\mathbf{j}(\w,\alpha)=i)=w^i$ and,
conditionally on $\{\mathbf{j}(\w,\alpha)=i\}$, the random variable $\beta(\w,\alpha)$ is
uniform on $[0,1)$.

Let us define the mapping $X_{f}^{*}$ as
\[
X_f^{*}(\x,\w,\alpha):=
\sum_{i=1}^N \bm{1}_{\{\mathbf{j}(\w,\alpha)=i\}}\,
g^i\!\left(\x,\w,\beta(\w,\alpha)\right).
\]
To conclude, we need to show that
$\big(X_f^{*}(\x,\w,\cdot),\,x^{\j(\w,\cdot)}\big)$ has joint distribution
$\pi_{\x,\w}$ for $\alpha\sim\unif[0,N)$. Take a Borel set $A\subseteq\RR^d$
and $j\in\{1,\dots,N\}$, we have
\begin{align*}
\mathbb{P}\Big(X_f^{*}(\x,\w,\alpha)\in A, &\,x^{\j(\w,\alpha)}=x^j\Big)
\\
&=
\mathbb{P}\Big(X_f^{*}(\x,\w,\alpha)\in A \,\Big|\, \j(\w,\alpha)=j\Big)
\,\mathbb{P}\big(\j(\w,\alpha)=j\big)\\
&=\mathbb{P}\Big(g^j\big(\x,\w,\beta(\w,\alpha)\big)\in A \,\Big|\, \j(\w,\alpha)=j\Big)\,w^j\\
&= G^j(\x,\w,A)\,w^j =\frac{\pi_{\x,\w}(A\times\{x^j\})}{\pi_{\x,\w}(\RR^d\times\{x^j\})}\,w^j.
\end{align*}
Finally, since the second marginal of $\pi_{\x,\w}$ is $\mu_{\x,\w}$, we have
$\pi_{\x,\w}(\RR^d\times\{x^j\})=w^j$, with $w^j>0$, hence
\[
\mathbb{P}\Big(X_f^{*}(\x,\w,\alpha)\in A,\,x^{\j(\w,\alpha)}=x^j\Big)
= \pi_{\x,\w}(A\times\{x^j\}),
\]
which shows that the joint law of
$\big(X_f^{*}(\x,\w,\alpha),\,x^{\j(\w,\alpha)}\big)$ is $\pi_{\x,\w}$.
In particular,
\rev{\[
\mathbb{E}\left[ c \Big(X_f^{*}(\x,\w,\alpha), x^{\j(\w,\alpha)}\Big)^p\right]
=\int c(x,y)^p\,\pi_{\x,\w}(\d x,\d y)
=\left(\W_p^c\big(f,\mu_{\x,\w}\big)\right)^p.
\]}
\end{proof}

\subsection{Limit as $N \to \infty$} \label{sec:limit:N}

First, we collect an estimate on the moments evolution of the discretized PDE \eqref{eq:euler}. Then, we will show that $f^N_n$ converges to $f_n$ using a Gr\"onwall-type argument.

\begin{lemma}[Moments evolution] \label{l:moments}
\rev{Let Assumption \ref{asm:F} hold for $\F$, and $(f_n)_{N\in \mathbb{N}}$ be defined by \eqref{eq:euler} with $f_0 \in \mathcal{P}_q(\Rd)$}.  There exists a constant $\tilde C = \tilde C(q)>0$ such that for any $n\geq 0$ it holds $M^{1/q}_q(f_n) \leq   \tilde{C} e^{ \tilde{C} \kappa_\F n\tau} ( M^{1/q}_q(f_0) + \sigma)$, \rev{with $\kappa_\F = \overline{\F}/\underline{\F}$}.
\end{lemma}
\begin{proof}
From \eqref{eq:euler} we obtain
\[
M_q(f_{n+1})=(1-\tau)M_q(f_n)+\tau\langle |\,\cdot\,|^q,Q^+(f_n,f_n)\rangle.
\]
Using the jump definition \eqref{eq:coll} and the inequality 
$
|x'|^q\le  C (\sigma^q + |x|^q+|x_*|^q)
$
(valid for $\gamma\in[0,1]^d$ and any additive noise with finite $q$-moment), we have
\[
\langle |\,\cdot\,|^q,Q^+(f_n,f_n)\rangle
\le  C \int_{\Rd}(\sigma^q + |x|^q)\,\frac{\F(x)}{\langle \F,f_n\rangle}\,f_n(\d x) \leq  C \kappa_{\F}(\sigma^q + M_q(f_n))\,.
\]
thanks to \rev{Assumption \ref{asm:F} on $\F$}. Hence
$
M_q(f_{n+1}) \leq \left (1+ \tau C\kappa_{\F} \right ) M_q(f_n) + \tau C\kappa_{\F}\sigma^q .
$
Iterating the inequality
yields
\[
M_q(f_n)
\le (1+C\kappa_\F\tau)^n M_q(f_0)
+ C\kappa_\F \sigma^q \tau \sum_{k=0}^{n-1}(1+C\kappa_\F\tau)^k.
\]
Using the identity for the geometric sum $\sum_{k=0}^{n-1}(1 + \tau a)^k  = ((1 + \tau a)^n - 1)/(a\tau)$ and $(1+a\tau)^n\le e^{a n\tau}$, we obtain
\[
M_q(f_n)\le e^{C\kappa_\F n\tau}\bigl(M_q(f_0)+\sigma^q\bigr),
\]
and therefore
$
M_q(f_n)^{1/q}\le \tilde{C}e^{\tilde{C}\kappa_\F n\tau}\left(M_q(f_0)^{1/q}+\sigma\right)
$ for some $\tilde C = \tilde C(q)$.
\end{proof}

Next, we apply Lemma \ref{l:coupling} in the construction of the nonlinear system \eqref{eq:iid}, by setting
\[ X^*_n(\cdot) := X^*_{\tilde{f}_n}(\X_n, \w_n,\cdot) \qquad \textup{with} \quad \tilde{f}_n  := \frac{\F f_n}{\langle \F, f_n\rangle}\,. \]
Through the random variables $(\gamma^i_n, \xi_n^i, \tau_n^i, \alpha^{i,1}_n, \alpha^{i,2}_n)$, we have coupled particle $i$ of the algorithm \eqref{eq:nanbu} with particle $i$ of the nonlinear system \eqref{eq:iid}. 
\rev{Recall the crossover-mutation update \eqref{eq:coll} is Lipschitz by with respect to $\D$ by Lemma \ref{l:lipschitz}.
By coupling optimality of $X^*_n$, we can track their distance over time
as
\begin{align}
\mathbb{E}\D(X^i_{n+1} ,\OX_{n+1}^i)  
 & \leq ( 1- \tau) \mathbb{E}\D(X^i_n ,\OX_n^i)  \notag\\
 & \qquad +  \tau\left( \mathbb{E}\D\left(X^{\j(\w_n,\alpha^{i,1}_n)}_n , X^*_{n}(\alpha^{i,1}_n) \right) + \mathbb{E}\D \left(X^{\j(\w_n,\alpha^{i,2}_n)}_n, X^*_{n}(\alpha^{i,2}_n) \right) \right)\notag \\
 & \leq (1 - \tau)  \mathbb{E}\D (X^i_n, \OX_n^i) + 2\tau \mathbb{E} \WW \Big(\tilde{f}_n, \sum_{i=1}^N w_n^i \delta_{X^i_n} \Big) \label{eq:est1}
\end{align}
Recall the weights are such that $w^i_n \propto \F(X^i_n)$, and so $\tilde{f}_n^N:= \sum_i w^i_n \delta_{X^i_n} = \F f^N_n/\langle \F, f^N_n\rangle$ is also a re-weighted probability measure like $\tilde f$. Thanks to Lemma \ref{l:stability} we then have 
 $
 \W^{|\cdot|}_1(\tilde{f}_n, \tilde{f}_n^N) \leq C_\F  \W^{|\cdot|}_1(f_n, f_n^N) \,.
 $
By continuing the estimate \eqref{eq:est1}, we  obtain 
\begin{align*}
\mathbb{E}\D(X^i_{n+1} , \OX_{n+1}^i) &\leq ( 1- \tau) \mathbb{E}\D(X^i_{n}, \OX_{n}^i) + 2\tau C_\F \mathbb{E}\WW(f_n, f^N_n)\\
&\leq ( 1- \tau) \mathbb{E}\D(X^i_{n} ,\OX_{n}^i) + 2\tau C_\F\mathbb{E} \WW(f_n, \overline{f}^N_n) + 2\tau C_\F \mathbb{E}\WW(\overline{f}^N_n, f^N_n)\,.
\end{align*}}%
where we also applied the triangle inequality with the empirical measure $\overline{f}^N_n$ associated with the nonlinear system $\mathbf{\OX}_n$.

Since \rev{$ \WW(f^N_n, \overline{f}_n^N) \leq (1/N)\sum_i \D(X^i_{n+1} ,\OX_{n+1}^i)$}, when we sum the estimate among all $i = 1,\dots,N$, we obtain
\rev{\[
\frac1N\sum_{i=1}^N \mathbb{E}\D(X^i_{n+1}, \OX_{n+1}^i)  \leq (1 + \tau(2C_\F - 1))\frac1N\sum_{i=1}^N \mathbb{E}\D(X^i_{n} , \OX_{n}^i) + 2\tau C_\F\mathbb{E}  \WW(\overline{f}^N_n,f_n)\,.
\]}%
It is crucial now to note that the term \rev{$\mathbb{E} \WW( \overline{f}^N_n,f_n)$ is something that we can control thanks to  $\WW\leq \W^{|\cdot|}_1$, and for $\W^{|\cdot|}_1$ we can use} the concentration inequality provided by Lemma \ref{l:rate} since $\overline{f}_n^N$ is made of i.i.d $f_n$-distributed particles. In particular, in view of the moments estimate (Lemma \ref{l:moments}), we have for constants $C = C(q,d, M_q(f_0)^{1/q},\sigma)>0$, $C' = C'(q)>0$
\[
\mathbb{E}\rev{ \W^{|\cdot|}_1}( \overline{f}^N_n,f_n) \leq C  e^{ C' \kappa_\F n \tau } \ve_1(N)\,.
\]

We iterate now the estimates for $n\geq 0$. To simplify the notation, let 
\rev{$
E_n := \frac1N\sum_{i=1}^N \mathbb{E}\D(X^i_{n}, \OX_{n}^i)\, .
$}
Then the previous inequality implies
\begin{align*}
E_{n+1}& \leq (1 + 2\tau \rev{C}_\F) \,E_n + 2\tau\rev{C_\F\,\mathbb{E}\WW(f_n,\overline f_n^N)} \\
& \leq (1 + 2\tau \rev{C}_\F)\,E_n + C_0\,\tau\,\rev{C_\F}\, e^{C' C_\F n\tau}\ve_1(N),
\end{align*}
for again, some constant \rev{$C_0=C_0(q,d, M_q(f_0)^{1/q},\sigma)>0$}. Iterating for $n \tau \leq T$, and using that $E_0 = 0$ yields the Gr\"onwall-type inequality
\rev{\begin{align*}
E_n 
& \leq (1 + 2\tau C_\F)^n E_0 
+ C_0\,\tau\, C_\F\, \ve_1(N)\sum_{k=0}^{n-1} (1 + 2\tau C_\F)^{\,n-1-k} e^{C' C_\F k\tau} \\
& \leq C_0\,\tau\,C_\F\, \ve_1(N)\sum_{k=0}^{n-1} e^{2C_\F(n-1-k)\tau}\, e^{C' C_\F k\tau} \\
& \leq C_0\,\tau\,C_\F\, \ve_1(N)\sum_{k=0}^{n-1} e^{C_\F\max(2, C')(n-1)\tau}   \leq  C_0 C_\F \, T\, e^{C_\F\max(2, C') T}\, \ve_1(N).
\end{align*}}%

Finally, since \rev{$\mathbb{E}\WW(f^N_n, \overline{f}^N_n) = \mathbb{E}\| f^N_n - \overline{f}_n^N\|_\BL \leq E_n$} by definition of the Wasserstein distance, we obtain the estimate
\rev{\begin{equation} \label{eq:est2}
\begin{split}
\mathbb{E}\|f^N_n - f_n \|_\BL &\leq \mathbb{E}\|f^N_n - \overline{f}^N_n \|_\BL + \mathbb{E}\|\overline{f}^N_n - f_n \|_\BL  \\
& \leq C T e^{C' T} \ve_1(N) 
\end{split}
\end{equation}}
for some  $C = C(\F, q,d, M_q(f_0)^{1/q},\sigma)>0$, $C' = C'(\F,q)>0$.

\subsection{Limit as $\tau \to 0$  and final estimate} Now, we show  that the explicit time discretization $f_n$ \eqref{eq:euler} converges to the solution of the PDE $f(t_n)$ \eqref{eq:boltzmann} as $\tau \to 0$ in the \rev{BL norm $\| f_n - f(t_n)\|_{\BL}$}.  We expect an explicit Euler discretization to return an error of order $\mathcal{O}(\tau)$. The proof is standard, \rev{see for instance \cite{piccoli2014generalized}}, but we give a brief sketch for completeness.

We note that the RHS of the Boltzmann-like equation is Lipschitz with respect to \rev{$\|\cdot\|_\BL$ since, thanks to crossover-mutation stability (Lemma \ref{l:lipschitz}) and selection stability (Lemma \ref{l:stability}), we have for any $f, g \in \mathcal{P}(\Rd)$
\[
\|  Q^+(f, f) -  Q^+(g,g)\|_\BL \leq 2 \left \| \frac{\F f}{\langle \F, f\rangle } - \frac{\F g}{\langle \F, g\rangle }   \right\|_\BL \leq 2 C_\F \|f - g \|_\BL\, .
\]
}%
\rev{Using moment bounds for $f$ (similarly to the time-discrete result Lemma \ref{l:moments}), it is possible to obtain a uniform estimate for the solution $f$ of the type
\[
\|f(t)-f(s)\|_{\BL}
\le C\,(M_1(f_0)+ \sigma)\,e^{\tilde C\,\kappa_\F T}\,|t-s|.
\]}%
Let $\rev{e_n := \|f_n - f(t_n)\|_{\BL}}$. By the weak formulation of \eqref{eq:boltzmann}, we have \rev{for $\phi$, $\textup{Lip}(\phi)\leq 1$, $\|\phi\|_\infty\leq 1/2$}
\begin{multline*}
\langle \phi, f_{n+1}-f(t_{n+1})\rangle
= (1-\tau)\langle \phi, f_n-f(t_n)\rangle
+ \tau \langle \phi, Q^+(f_n,f_n)-Q^+(f(t_n),f(t_n))\rangle \\
\quad + \int_{t_n}^{t_{n+1}}\Big(\langle \phi, Q^+(f(t_n),f(t_n)) - Q^+(f(s),f(s))\rangle
+ \langle \phi, f(s)-f(t_n)\rangle\Big)\,\d s.
\end{multline*}
Taking the supremum over $\phi$, and using the Lipschitz stability of $Q^+$ in \rev{$\|\cdot\|_{\BL}$}, we deduce
\rev{\begin{align*}
e_{n+1}
&\le (1-\tau)e_n + \tau 2C_\F e_n
+  \int_{t_n}^{t_{n+1}}\Big(2C_\F \|f(t_n)-f(s)\|_{\BL} + \|f(s)-f(t_n)\|_{\BL}\Big)\,\d s \\
&\le \big(1+\tau(2C_\F-1)\big)e_n
+ (1+2C_\F)\int_{t_n}^{t_{n+1}}\|f(s)-f(t_n)\|_{\BL}\,\d s \\
&\le \big(1+\tau(2C_\F-1)\big)e_n
+ (1+2C_\F)\int_{t_n}^{t_{n+1}} \tilde{C}\,(1+2C_\F)\,e^{\tilde C\,2C_\F T} (s-t_n)\,\d s \\
&\le \big(1+\tau(2C_\F-1)\big)e_n + \tilde{C}\,(1+2C_\F)^2\,e^{\tilde C'\,2C_\F T}\tau^2\,,
\end{align*}}
for some $\tilde{C} = \tilde{C}(q,M_q(f_0)^{1/q},\sigma)>0$, $\tilde{C}' = \tilde{C}'(q)$.

Iterating the above inequality for $t_n=n\tau\le T$ and using $e_0=0$ gives for some other constants \rev{$C = (\F, q,M_q(f_0)^{1/q},\sigma)$, $C' = C'(\F, q)>0$ }
\begin{equation} \label{eq:est3}
\sup_{t_n\le T} \|f_n-f(t_n)\|_{\rev{\BL}} \le C  \tau e^{C' \kappa_\F T},
\end{equation}
which shows that $\|f_n-f(t_n)\|_{\rev{\BL}}=\mathcal{O}(\tau)$ uniformly on $[0,T]$.

\begin{proof}[Proof of Theorem \ref{t:main}]
To recap, we have compared in Section \ref{sec:limit:N} the empirical measure $f^N_n$ generated by Algorithm \ref{alg:ga} with the time-discretized PDE \eqref{eq:euler}. In this section, instead, we have compared $f_n$ with a measure solution to the PDE \eqref{eq:boltzmann}. By collecting the two  estimates, \eqref{eq:est2} and \eqref{eq:est3}, we obtain the desired error decomposition
\rev{\begin{align*}
\sup_{t_n = n\tau \le T} \|f^N_n-f(t_n)\|_{\BL} &  \leq \sup_{t_n = n\tau \le T} \|f^N_n-f_n \|_{\BL}  + \sup_{t_n = n\tau \le T} \|f_n-f(t_n)\|_{\BL} \\
& \leq C T e^{C' T} \left( \,\ve_1(N)\, + \,\tau\,\right)
\end{align*}
where $C = C(\F, d,q,M_q(f_0)^{1/q},\sigma)>0$,  $C' =  C(\F,q)$ some positive constants.}
\end{proof}

\section{Outlook}

In this paper, we restricted the attention to fitness-based selection. Another important strategy in genetic algorithms is rank-based selection, in which only the top-performing solutions are chosen for reproduction \cite{khalid2013selection}. Rank-based particle interactions have been studied, for instance, in \cite{jourdain2013chaos, bencheikh2022weak} in connection with mean-field-type PDEs. The analysis of rank-based selection within generalized Boltzmann models such as \eqref{eq:boltzmann} remains, to our knowledge, an open problem.

\section*{Acknowledgment}
The author was supported by the Wolfson Fellowship of the Royal Society “Uncertainty quantification, data-driven simulations and learning of multiscale complex systems governed by PDEs” of Prof. L. Pareschi at Heriot-Watt University.

\bibliographystyle{abbrv}
\bibliography{bibfile}
\end{document}